\documentclass[reqno,11pt]{amsart}

\usepackage{epsf}
\usepackage{amssymb,graphics,graphicx,wrapfig}
\usepackage{amsmath, relsize} 
\usepackage{todonotes} 
\usepackage{amsthm} 
\usepackage{amsfonts}
\usepackage{bbm}
\usepackage{amssymb}
\usepackage{array}
\usepackage{url}
\usepackage{xcolor}
\usepackage{hyperref}
\usepackage{graphicx} 
\usepackage{caption}
\usepackage{subcaption}
\usepackage{enumitem}
\usepackage{varwidth}

\def\eps{\varepsilon}

\def\be{\begin{equation}}
\def\ee{\end{equation}}
\def\ba{\begin{align}}
\def\bm{\begin{multline}}
\def\bfig{\begin{figure}[htb]}
\def\efig{\end{figure}}

\numberwithin{equation}{section}
\newtheorem{theorem}{Theorem}[section]
\newtheorem{proposition}[theorem]{Proposition}
\newtheorem{lemma}[theorem]{Lemma}

\newtheorem{definition}{Definition}

\DeclareMathSymbol{\leqslant}{\mathalpha}{AMSa}{"36}
\DeclareMathSymbol{\geqslant}{\mathalpha}{AMSa}{"3E}
\DeclareMathSymbol{\doteqdot}{\mathalpha}{AMSa}{"2B}
\DeclareMathSymbol{\circlearrowright}{\mathalpha}{AMSa}{"08}
\DeclareMathSymbol{\subsetneq}{\mathalpha}{AMSb}{"28}
\DeclareMathSymbol{\supsetneq}{\mathalpha}{AMSb}{"29}
\renewcommand{\leq}{\;\leqslant\;}
\renewcommand{\geq}{\;\geqslant\;}

\newcommand{\dd}{{\rm d}}
\newcommand{\e}[1]{\,{\rm e}^{#1}\,}

\newcommand{\upchi}{\raise 2pt \hbox{$\chi$}}

\def\writefig#1 #2 #3 {\rlap{\kern #1 truecm \raise #2 truecm
\hbox{#3}}}

\newcommand{\caC}{{\mathcal C}}

\newcommand{\caF}{{\mathcal F}}

\newcommand{\caH}{{\mathcal H}}
\newcommand{\caI}{{\mathcal I}}

\newcommand{\caL}{{\mathcal L}}

\newcommand{\caN}{{\mathcal N}}
\newcommand{\caO}{{\mathcal O}}

\newcommand{\caS}{{\mathcal S}}

\def\bbone{{\mathchoice {\rm 1\mskip-4mu l} {\rm 1\mskip-4mu l} {\rm 1\mskip-4.5mu l} {\rm 1\mskip-5mu l}}}

\newcommand{\bbE}{{\mathbb E}}

\newcommand{\bbN}{{\mathbb N}}

\newcommand{\bbP}{{\mathbb P}}

\newcommand{\bbR}{{\mathbb R}}

\newcommand{\bbZ}{{\mathbb Z}}

\newcommand{\bsgamma}{{\boldsymbol\gamma}}

\begin{document}


\title{Interacting self-avoiding polygons}

\author{Volker Betz,  Helge Sch\"afer, Lorenzo Taggi}
\address{Volker Betz \hfill\newline
\indent FB Mathematik, TU Darmstadt \hfill\newline
{\small\rm\indent http://www.mathematik.tu-darmstadt.de/$\sim$betz/} 
}
\email{betz@mathematik.tu-darmstadt.de}
\address{Helge Sch\"afer \hfill\newline
\indent FB Mathematik, TU Darmstadt \hfill\newline
{\small\rm\indent https://www.mathematik.tu-darmstadt.de/fb/personal/details/helge\_schaefer.de.jsp} 
}
\email{hschaefer@mathematik.tu-darmstadt.de}

\address{Lorenzo Taggi \hfill\newline
\indent FB Mathematik, TU Darmstadt \hfill\newline
{\small\rm\indent https://sites.google.com/site/lorenzotaggiswebpage2} 
}
\email{taggi@mathematik.tu-darmstadt.de}

\begin{abstract}
We consider a system of  self-avoiding polygons interacting through a potential that penalizes or rewards the number of mutual touchings and we provide an exact computation of the critical curve separating a regime of long polygons
from a regime of localized polygons.
Moreover, we prove the existence of a sub-region of the phase diagram where the self-avoiding polygons 
are space filling and we provide a non-trivial characterization of the regime where the polygon length
admits uniformly bounded exponential moments.
\end{abstract}

\maketitle

\section{Introduction}
\label{introduction}
Let $\Lambda_L$ be graph corresponding to a torus with vertex set  $[-L/2,L/2)^d \cap \bbZ^d$, for $d \geq 2$,
and whose edges connect nearest neighbour vertices. 
The edges are directed and any two neighbour vertices are connected by two parallel edges having opposite orientation.
A (directed) self-avoiding polygon in $\Lambda_L$ is a connected, directed subgraph of $\Lambda_L$, where each vertex has one ingoing and one outgoing edge.
This paper considers a system of random interacting self-avoiding polygons (ISAP).
This system can be seen as a model for a random polymer in a random environment which is constituted of an ensemble of 
other polymers.
Polymers in random media have been investigated in various settings   (see \cite[Chapter 12]{HollanderNotes} for a review) and are of great physical interest.
Moveover, ISAP  interpolates between two paradigmatic statistical mechanics models which behave qualitatively similar in more than three dimensions, but which have been conjectured to belong to different universality classes, allowing thus to compare them.
The first of these models is the single self-avoiding polygon and was introduced as a simple model for a  polymer, see \cite{Madras} for an overview.
The second model, {random lattice permutations} \cite{ Betz2},
 can be viewed as a toy model for Bose-Einstein condensation \cite{Feynman, Kikuchi}.
 The presence of hard-core interaction makes random lattice permutations a very difficult model to treat mathematically.
 Contrary to random lattice permutations, ISAP does not present such a rigidity and it is thus better suitable for rigorous mathematical investigation.

\section{The model of interacting self-avoiding polygons}
\label{The model}
\subsection{Definitions and results}
Let ${\rm SAP}_{x,n}$ denote the set 
of self-avoiding polygons in $\Lambda_L$ covering $n$ vertices and
containing the vertex $x \in \Lambda_L$. Note that 
${\rm SAP}_{x,n} = \emptyset$ if $n$ is odd and 
${\rm SAP}_{x,2}$ contains $2d$ polygons.
Let $\zeta$ denote the empty polygon, which has no vertices, and define  ${\rm SAP}_{x,0}  : = \{ \zeta \}$.
For any $x \in \Lambda_L,$ we define
${\rm SAP}_{x} = \bigcup_{n \geq 0} {\rm SAP}_{x,n}$. 
For $\gamma \in {\rm SAP}_x$, we write 
$\| \gamma \|$ for the number of edges in $\gamma$.

An important classical 
fact about self-avoiding polygons is the existence of the 
(dimension-dependent) {\em connective constant}: if 
${\rm SAP}(n)$ denotes the set of all self-avoiding polygons on $\mathbb{Z}^d$ covering $n$ vertices and containing the origin,
\begin{equation} \label{cc}
\mu = \lim_{n \to \infty} \big( |{\rm SAP}(n)| \big)^{1/n} 
\end{equation}
exists, and indeed  
\begin{equation}
	\label{connective constant inequality}
	|{\rm SAP}(n)| \leq 2 \, (d-1) n \mu^n
\end{equation}
for all $n \in \bbN$, see \cite{Madras}. Since $|{\rm SAP}_{x,n}| 
\leq |{\rm SAP}(n)|$, the latter inequality also holds for 
${\rm SAP}_{x,n}$. 

For given $\Lambda_L$, the probability space for the 
model of interacting self-avoiding polygons (ISAP) is $\Omega = 
\prod_{x \in \Lambda_L} {\rm SAP}_x$. For 
$\bsgamma = (\gamma_x)_{x \in \Lambda_L} \in \Omega$, we define 
the cumulative length of $\bsgamma$ as
\[
\caC(\bsgamma) = \sum_{x \in \Lambda_L} \| \gamma_x \|,
\]
and the total overlap of $\bsgamma$ as 
\[
\caI(\bsgamma) = \sum_{x \in \Lambda_L} \Big( \big( \sum_{y \in \Lambda_L}
\bbone\{x \in \gamma_y\} - 1 \big) \vee 0 \Big).
\]
In words, $\caI(\bsgamma)$ counts the total value of overlaps of (non-empty) self-avoiding polygons in the system, where an 
overlap 
at a site $x$ has value $(k-1) \vee 0$ if $k$ 
self-avoiding polygons contain $x$. The energy of a given 
$\bsgamma \in \Omega$ is defined by 
\begin{equation}
	\label{Hamiltonian alpha lambda}
	\caH_{L,\alpha,\lambda}(\bsgamma) := \alpha \caC(\bsgamma)
	+ \lambda \caI(\bsgamma) + 
	\sum_{x \in \Lambda_L: \| \gamma_x\| > 0} 
	\log \|\gamma_x\|,
\end{equation}
and the probability of $\bsgamma$ is as usual given by 
\[
\bbP_{L,\alpha,\lambda}(\bsgamma) = \frac{1}{Z(L,\alpha,\lambda)} 
\e{-\caH_{L,\alpha,\lambda}(\bsgamma)}.
\]
The reason for the sum of logarithms in $\caH_{L,\alpha,\lambda}$ 
is that this guarantees convergence of the finite volume model of 
ISAP  to the finite volume model of nearest neighbour 
spatial random permutations in the limit $\lambda \to \infty$, 
which, as discussed in the introduction, motivates a lot of the 
interest in the current model.

Our first result is an exact computation of  the critical curve separating the regime where the expected polygon length is uniformly bounded from the regime where the expected polygon length grows to infinity with $L$.
We define 
\[
\caL(\alpha,\lambda) := \limsup_{L \to \infty} \bbE_{L,\alpha,
\lambda}(\| \gamma_o \|),
\]
where  $o$ is used to denote the origin.  
\begin{figure}
\centering
\includegraphics[scale=0.42]{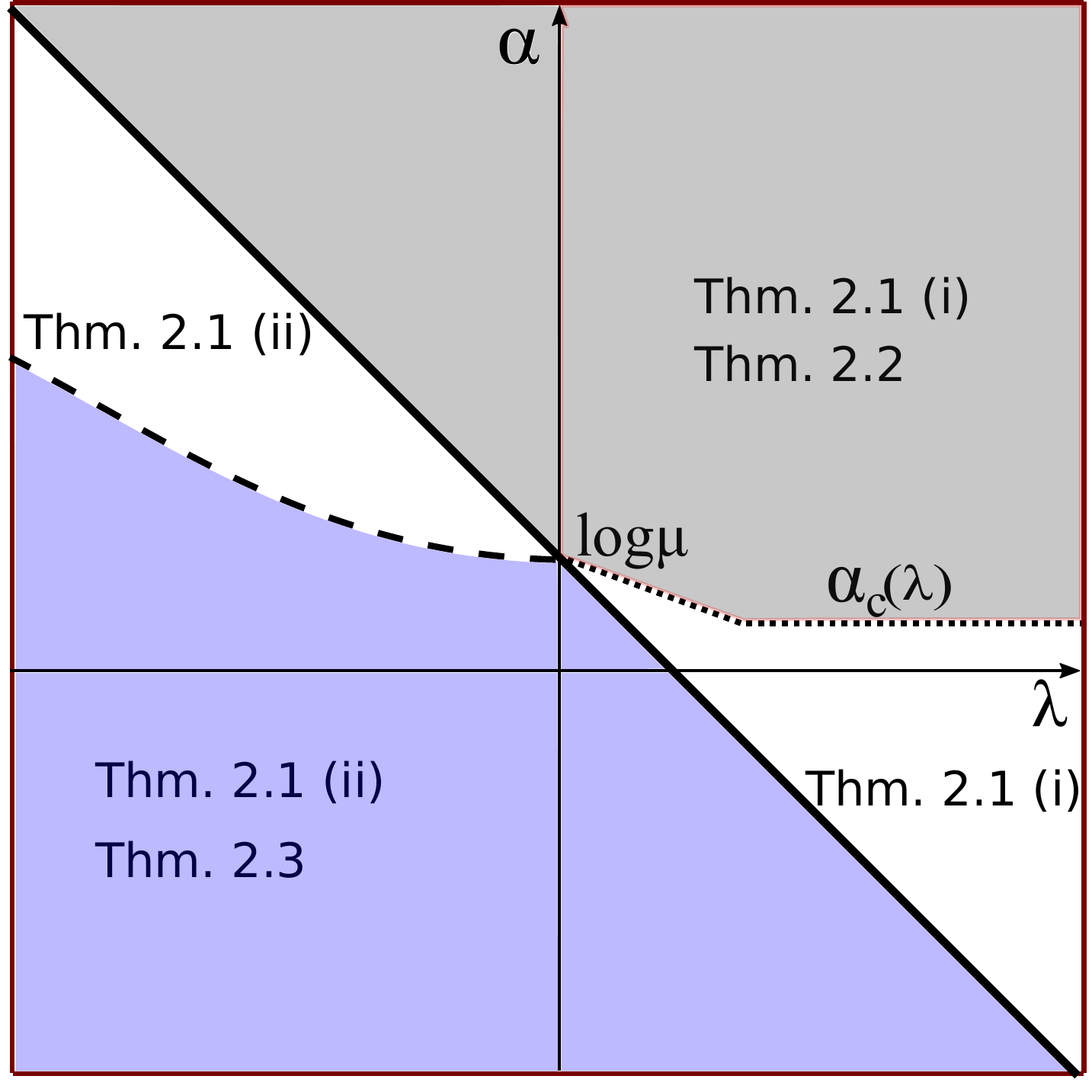}
\caption{The continuous line represents the critical curve $\alpha = \log \mu - \lambda$. The first (resp. second) statement in Theorem \ref{cumCycLen} holds in the whole region on  the right (resp. left) of the critical curve, Theorem \ref{expdecay} (resp. Theorem \ref{theo:weakly space filling}) holds in the darker region on the right (resp. left) of the critical curve.
}
\label{Fig:phasediagram}
\end{figure}
\begin{theorem} \label{cumCycLen}
Let $\mu$ denote the connective constant defined in \eqref{cc}.\\
(i): If $\alpha + \lambda > \log \mu$, then 
$\caL(\alpha,\lambda) < \infty$. \\
(ii): If $\alpha + \lambda < \log \mu$, then 
$\caL(\alpha,\lambda) = \infty$.
\end{theorem}

Our further results investigate both of the cases above a bit 
more. First, we show that there exists a subset 
of the region where the statement (i) of Theorem 
\ref{cumCycLen} holds, such that for all 
$(\alpha,\lambda)$ in this subset, the quantity 
$\Vert\gamma_o\Vert$ is exponentially integrable.
See Figure \ref{Fig:phasediagram}.

\begin{theorem}\label{expdecay}
	There exists a decreasing function $\lambda \mapsto \alpha_{\rm c}
	(\lambda)$ such that for all $\lambda \in \bbR$  
	and all $\alpha > \alpha_{\rm c}(\lambda)$, there exists 
	$\delta > 0$ such that 
	\begin{equation}\label{expintegrable}
	\limsup_{L \to \infty} \bbE_{L,\alpha,\lambda}
	(\e{\delta \|\gamma_o\|}) < \infty.
	\end{equation}
	The function $\alpha_{\rm c}$ is such that 
	$\alpha_{\rm c}(\lambda) = 
	\log \mu - \lambda$ when $\lambda \leq 0$ and 
	$\alpha_{\rm c}(\lambda) = (\log \mu - c_1 \, \lambda) \vee c_2$ when $\lambda > 0$, where the constants $c_1 \in (0, \infty)$ and $c_2 \in (- \infty, \log \mu)$ are given in \eqref{alphaCformula}.
\end{theorem}
For $\lambda  > 0$ and $\alpha \geq \log \mu$, 
 (\ref{expintegrable}) follows quite directly  from  the trivial upper bound, $\mathbb{P}_{L, \alpha, \lambda}( \|\gamma_o \| > k) \leq e^{- \alpha k}$,
and from (\ref{cc}). Intuitively, this entails that the 
mutual repulsion on polygons that is implemented through 
$\lambda > 0$ does not make the polygons much more likely 
to be long. 
Theorem \ref{expdecay} goes beyond this and states that 
any positive repulsion $\lambda>0$ leads to exponential integrability 
of $\| \gamma_o\|$ for values of $\alpha$ where without
repulsion (i.e.\ for $\lambda = 0$) the cycles would be 
long. Comparison with numerical results about spatial 
random permutations (cf. Subsection \ref{discussion}) 
below suggests that there is a nonempty region of the 
parameter space so that Theorem \ref{cumCycLen} (i) holds
but exponential integrability of $\| \gamma_o\|$ fails; 
Theorem \ref{expdecay} just gives an upper bound on how 
large such a region could be - to prove (or disprove) the 
existence of such a region is a major open problem in the 
model. 

Turning to the region of phase space where Theorem 
\ref{cumCycLen} (ii) holds, 
we show in Theorem \ref{theo:weakly space 
filling} below that there exists a subset of this region 
where  $\gamma_o$ is weakly space 
filling; see Figure \ref{Fig:phasediagram}. Again, a 
natural question is whether Theorem \ref{theo:weakly space 
filling} actually holds for the full half space where 
Theorem \ref{cumCycLen} (ii) is valid. We do not have a 
clear conjecture about whether or not this  should be
case. 

In order to state our result precisely, 
let $\Gamma_{L}^{\xi}(\gamma_o)$ be the set of vertices in $\Lambda_L$ which have distance at least $\xi$ from
$\gamma_o$ and let $A^\xi_L(\gamma_o)$ be the largest connected component of $\Gamma_{L}^{\xi}(\gamma_o)$. 

\begin{theorem}\label{theo:weakly space filling}
For any $\alpha \in \mathbb{R}$, there exists $\lambda_{sf} = \lambda_{sf}(\alpha)$ such that if $\lambda < \lambda_{sf} $,
then the polygons are  weakly space filling, i.e.
there exist two constants $c, \xi \in (0, \infty)$ such that,
\begin{equation}\label{eq: weakly space filling}
\lim\limits_{  L \rightarrow \infty} \, \, \mathbb{P}_{L, \alpha, \lambda} ( |A^{\xi}_L(\gamma_o)| > c \log L    ) =0.
\end{equation}
Moreover, if $\alpha < \log \mu$, we can choose $\lambda_{sf}(\alpha)  = \log \mu - \alpha > 0.$
\end{theorem}

A corresponding result has been obtained in \cite{Copin} 
for a single self-avoiding polygon, and indeed our proof 
heavily uses their methods and results.

\subsection{Discussion of the model}
\label{discussion}
The main motivation of the model is its relation to a variant of 
the model of spatial random permutations (SRP). 
The latter is defined on 
the set $\caS(\Lambda_L)$ of permutations of $\Lambda_L$, and the 
probability of a permutation $\sigma \in \caS(\Lambda_L)$ is given by 
\[
\bbP_{L,\alpha}(\{\sigma\}) = \begin{cases}
	0 & \text{if } |\sigma(x) - x| > 1 \text{ for some } x, \\
	\frac{1}{Z(L,\alpha)} \e{- \alpha \sum_{x \in \Lambda_L} |\sigma(x) - x|} & 
	\text{otherwise}.
	\end{cases}
\]
Here, the distance function 
respects the periodic boundary conditions. The most important 
question in the model of SRP is about the existence of long cycles: 
if $C_x(\sigma)$ denotes the cycle of $\sigma$ containing a 
point $x$, and $|C_x|$ its length, 
then the aim is to prove that there exists some 
$\alpha \in \bbR$ such that 
$\lim_{L \to \infty} \bbE_{L,\alpha}(|C_x|) = \infty$. This 
seems to be a rather difficult problem and has deep connections to 
the theory of quantum spin systems - see \cite{BU1, BU2, Feynman} 
for details. 

$\caS(\Lambda_L)$, the support of the measure of spatial random permutations, can be 
viewed as the set of (typically non-connected) 
directed subgraphs of $\Lambda_L$, where all 
vertices have degree $2$. On the other hand, each ISAP 
$\bsgamma$ with $\caI(\bsgamma) = 0$ can be mapped to such a
subgraph, by 'forgetting' to which point $x$ each self-avoiding 
polygon belongs. This also means that for given $\bsgamma$,
there are $\prod_{x \in \Lambda_L, \|\gamma_x\|>0} \|\gamma_x\|$ different ISAPs that lead to 
the same spatial random permutation as $\bsgamma$. 
Thus thanks 
to the sum of logarithms in \eqref{Hamiltonian alpha lambda}, 
we can relate ISAP to SRP in the limit of large $\lambda$; 
for example, for any $x \in \Lambda_L$ and any $k > 1$,  
\[
\lim_{\lambda \to \infty} \bbP_{L,\alpha,\lambda} 
(\max_{z \in \Lambda: x \in \gamma_z} \Vert \gamma_z \Vert = k) = 
\lim_{\lambda \to \infty} \bbP_{L,\alpha,\lambda} 
(\min_{z \in \Lambda: x \in \gamma_z} \Vert \gamma_z \Vert = k) = 
  \bbP_{L,\alpha}(|C_x| = k). 
\]
By defining a suitable projection from ISAP to multigraphs  
for finite $\lambda$, we could make this connection much stronger, 
but this is not of interest here. Nevertheless, already the given 
example shows that understanding ISAP for large (fixed) $\lambda$
in the limit $L \to \infty$ may be 
of interest for the study of SRP; on the other hand, of course the 
problem of interchanging the limits of $\lambda \to \infty$ and 
$L \to \infty$ is a serious one and would need to be addressed as 
well. 

We have chosen the representation \eqref{Hamiltonian alpha lambda}
for its obvious relations to SRP, but there is another set of parameters 
that reveals additional information about the model.    
The key observation is that on vertices that 
are not empty, the overlap is precisely one unit lower than the 
number of polygons that visit that vertex. Thus writing 
$\caO_x(\bsgamma)$ for the occupation number of (i.e.\ number of polygons visiting) a vertex $x$, we find that   
\[
\caI(\bsgamma) = \caC(\bsgamma) - \sum_{x \in \Lambda_L} 
\bbone \{ \caO_x(\bsgamma) \geq 1 \} = \caC(\bsgamma) - |\Lambda_L| + 
\caN(\bsgamma),
\]
with 
\[
\caN(\bsgamma) = \sum_{x \in \Lambda_L} \bbone \{ \caO_x(\bsgamma) = 0 \}
\]
being the number of empty vertices. Thus if we define 
\[
\tilde \caH_{L,\rho, \nu}(\bsgamma) = \rho \caC(\bsgamma) + \nu \caN(\bsgamma) + 
\sum_{x \in \Lambda_L, \| \gamma_x \| > 0} \log \| \gamma_x \| 
\]
and
\begin{equation} \label{main prob meas}
\tilde \bbP_{L,\rho,\nu} (\bsgamma) = \frac{1}{\tilde Z_{L,\rho,\nu}}
\e{- \tilde \caH_{L,\rho,\nu}(\bsgamma)},
\end{equation}
then 
\[
\bbP_{L,\alpha,\lambda}(\bsgamma)  = 
\tilde \bbP_{L, \alpha + \lambda, \lambda} (\bsgamma)
\]
for all $L,\alpha,\lambda,\bsgamma$. Note that our new set of parameters
makes the relation to another important model of interacting shapes 
apparent, namely the double dimer model \cite{Double Dimer}: indeed, 
in the limit $\nu \to \infty$, the finite volume model converges to the
double dimer model thanks to the sum of logarithms in the Hamiltonian. 
Again, the question of interchange 
of the limits in $L$ and $\nu$ in highly non-trivial.

Model \eqref{main prob meas} will be used to prove Theorem 
\ref{cumCycLen}. We note already here an interesting aspect about that 
theorem which becomes apparent from  \eqref{main prob meas}: 
since we trivially have $\caN(\bsgamma) 
\leq |\Lambda_L|$, in case (ii) of Theorem \ref{cumCycLen}
we find that 
\[
\left\vert\tfrac{1}{|\Lambda_L|} 
\tilde{\bbE}_{L,\rho,\nu} (\tilde \caH_{L,\rho,\nu})\right\vert \geq 
\tfrac{\vert\rho\vert}{|\Lambda_L|}  \tilde{\bbE}_{L,\rho,\nu} (\caC) - \vert\nu\vert = 
\vert\rho\vert \tilde{\bbE}_{L,\rho,\nu} (\| \gamma_o \| ) -\vert\nu\vert\to \infty
\]
as $L \to \infty$ if $\rho\neq 0$ (the last equality holds due 
translation invariance). In other words, the mean energy per unit 
volume diverges in the case $\alpha + \lambda < \log \mu$, which 
from the point of view of statistical mechanics 
makes the model highly singular in that regime. Note, however, that 
the cases of interest for the connection to SRP are those of 
large $\lambda$  and for fixed $\alpha$; those are always covered by 
case (i) of Theorem \ref{cumCycLen}.

\section{Proofs of the Theorems}
\subsection{Proof of Theorem \ref{cumCycLen}}
We will work with the model \eqref{main prob meas}.
The statement of Theorem \ref{cumCycLen} then is that 
\[
\limsup_{L\to\infty} \tilde \bbE_{L,\rho,\nu} (\caC / |\Lambda_L|)
\begin{cases} < \infty & \text{if } \rho > \log \mu \\
= \infty & \text{if } \rho < \log \mu.	
\end{cases}
\]
We begin with the second case.
We first consider the case $\nu=0$. Then, the model is 
non-interacting, and 
$\tilde \bbE_{L,\rho,0} (\|\gamma_o\|) = \bbE_{L,\rho}(\|\gamma\|)$, 
where the latter expectation is for the single self-avoiding 
polygon model with probability measure  
\[
\bbP_{L,\rho}(\gamma)  = \frac{1}{Z_{L,\rho}} 
\e{- \rho \| \gamma \|} \frac{1}{\|\gamma\| \vee 1}.
\]
For this measure without the factor of $\frac{1}{\|\gamma\| \vee 1}$ it
has been shown in \cite{Copin} that $\gamma$ is almost surely weakly 
space filling, and a fortiori we have 
$\limsup_{L \to \infty} \tilde  \bbE_{L,\rho,0} (\|\gamma_o\|) = \infty$. 
Inspection of the proofs in \cite{Copin} shows that the factor 
$\frac{1}{\|\gamma\| \vee 1}$ does not change this. 

In order to treat $\nu \neq 0$, we define the  finite volume 
pressure as 
\[
\Phi_L(\rho,\nu) = \frac{1}{|\Lambda_L|} \log \tilde Z_{L,\rho,\nu}.
\]
The usual argument involving H\"older's inequality (which we sketch in 
the appendix for the reader's convenience) shows that 
$\Phi_L$ is a convex function. Direct calculation shows that
\[
\tilde \bbE_{L,\rho,\nu} (\| \gamma_o \| ) = \frac{1}{|\Lambda_L|} 
\tilde  \bbE_{L,\rho,\nu} (\caC) = - \partial_\rho \Phi_L(\rho,\nu),
\]
and 
\[
\frac{1}{|\Lambda_L|}  \tilde  \bbE_{L,\rho,\nu}(\caN) = - \partial_\nu \Phi_L(\rho,\nu).
\]
The latter quantity is strictly between $0$ and $1$. Using this 
observation and the fundamental theorem of calculus, 
we find that for each $\nu \in \bbR$ 
\[
| \Phi_L(\rho,\nu) - \Phi_L(\rho,0) | = \Big| 
\int_0^\nu \partial_s \Phi_L(\rho,s) \, \dd s \Big| \leq |\nu|.
\]
The triangle inequality then yields 
\begin{equation}
\label{eq:compare Phi}
|\Phi_L(\rho,\nu) - \Phi_L(\rho+h, \nu)| \geq 
| \Phi_L(\rho,0) - \Phi_L(\rho+h,0) | - 2 \nu.
\end{equation}
By the convexity of $\Phi_L$, the nonnegative
function $\rho \mapsto -\partial_\rho \Phi_L(\rho,\nu) = \tilde  \bbE_{L,\rho,\nu}(\| \gamma_o\|)$ is 
monotone decreasing for each $\nu$.  
Thus the inequality \eqref{eq:compare Phi} implies 
\[
\Phi_L(\rho,0) - \Phi_L(\rho + h,0) \leq - \int_0^h \partial_r \Phi_L(\rho + r,\nu) \, \dd r 
+ 2 \nu. 
\]
In particular, if $\rho < \log \mu$ and 
$0 < h < \log \mu - \rho$, then both sides of the above inequality diverge to $+\infty$ as $L \to \infty$. So, 
again using monotonicity, we get 
\[
\infty = - \lim_{L\to\infty} \int_0^h \partial_r \Phi_L(\rho + r,\nu) \, \dd r 
\leq - h \lim_{L\to\infty} \partial_\rho \Phi_L(\rho)
= h \lim_{L\to\infty} \tilde \bbE_{L,\rho,\nu}(\|\gamma_o\|).
\]
This shows claim (ii) of Theorem \ref{cumCycLen}.  

Let us now treat case (i). We write 
$|\Lambda_L| = N$ for brevity, and observe that for 
all $k \in \bbN$, 
\[
\tilde  \bbP_{L,\rho,\nu}(\caC = k) \leq \e{|\nu| N} 
\tilde \bbP_{L,\rho,0}(\caC = k) = \e{|\nu| N} 
\sum_{n_1, \ldots, n_N: \sum_{i} n_i = k/2} 
\tilde  \bbP_{L,\rho,0} (\forall i: \| \gamma_{x_i} \| = 2 n_i). 
\]
The latter probability factorizes, and for each $i$, we have 
\[
\tilde \bbP_{L,\rho,0}( \| \gamma_{x_i} \| = 2 n_i) \leq 
|{\rm SAP}_{2n_i}| \e{-2 \rho n_i} / (2n_i).
\]
By \eqref{connective constant inequality}, $|{\rm SAP}_n| \leq 2 (d-1) n \mu^n$ for all $n$. 
The exponentials then combine to a global factor of 
$\e{- (\rho - \log \mu) k}$, and the number of ways of partitioning
the integers $1$ to $k/2$ into $N$ possibly empty subsets is given 
by $(k+1)^{N-1} / (N-1)!$. In conclusion, we obtain 
\begin{equation}
\label{eq:C estimate}
\tilde \bbP_{L,\rho,\nu}(\caC = k) \leq \e{N (|\nu| + \log (d-1))} 
\e{- k (\rho - \log \mu)} \frac{(k+1)^{N-1}}{(N-1)!}.
\end{equation}
Let us abbreviate $\eps := \rho - \log \mu$ and 
$q = |\nu| + \log (d-1)$. Estimate \eqref{eq:C estimate} 
is actually quite poor when  $k$ is of the order of 
$N / \eps$; in that case a calculation involving Stirlings 
formula shows that the right hand side grows like 
$\frac{1}{2\pi N} \e{N(q+2)} \eps^{-N}$ which, at least 
for small $\epsilon$, is much large than $1$ for large 
$N$. We should therefore use \eqref{eq:C estimate} only 
for even larger $k$. Observe that 
for each $l,m \in \bbN$, we have  
\begin{equation} \label{must be summable}
\tilde \bbE_{L,\rho,\nu} (\caC / |\Lambda_L|) = \frac{1}{N} 
\sum_{k=0}^\infty \tilde \bbP_{L,\rho,\nu} ( \caC > k) \leq 
\sum_{l=0}^\infty \tilde \bbP_{L,\rho,\nu}( \caC > l N)
\leq m + \sum_{l=m}^\infty \tilde \bbP_{L,\rho,\nu}( \caC > l N)
\end{equation}
For $j \geq 0$, estimate \eqref{eq:C estimate} gives
\[
\begin{split}
& \tilde  \bbP_{L,\rho, \nu}(jN < \caC \leq (j+1)N) \leq 
\e{qN} \sum_{k= jN+1}^{(j+1)N} \frac{(k+1)^{N-1}}{(N-1)!} 
\e{-\eps k}\\
& \quad \leq \e{qN} \frac{((j+1)N+1)^{N-1}}{(N-1)!} \e{-\eps j N} 
\sum_{k=1}^\infty \e{-\eps k} \\
& \quad \leq \e{N (q - \eps j)} (j+2)^{N-1} \frac{N^{N-1}}
{\sqrt{2 \pi} (N-1)^{N-1/2} \e{-N}} \frac{1}{1-\e{-\eps}} \\
& \quad \leq \e{N(q - \eps j + \log(j+2) + 1)} \frac{1}{\sqrt{2 N \pi} 
(1 - \e{-\eps})} (1 + \frac{1}{N-1})^{N-1}
\end{split}
\]
The product of the latter two factors above is $\leq 1$ for large $N$. So by 
choosing $m$ so large that 
\[
\eps j - \log (j+2) - q - 1 > \eps j/2 
\]
for all $j \geq m$, we find that for all large enough $L$ and $l\geq m$,  
\[
\tilde  \bbP_{L,\rho,\nu}( \caC > lN) \leq  
\sum_{j=l}^\infty \e{-\eps j /2} = \frac{\e{-\eps l/2}}{1 - \e{-\eps/2}}
.
\]
So, \eqref{must be summable} is summable and the proof is finished.

\subsection{Proof of Theorem \ref{expdecay}}
\label{sect:proof of Theo2}
When $\lambda \leq  0$ and $\alpha +\lambda > \log\mu$, \eqref{expintegrable} follows from 
$
\mathbb{P}_{L, \alpha, \lambda} ( \gamma_o  = \tilde{\gamma} ) \leq e^{ - ( \alpha + \lambda) \Vert \tilde{\gamma}\Vert}
$
and (\ref{cc}).
By Theorem \ref{cumCycLen}, we then have $\alpha_{\rm c}(\lambda) = \log \mu - \lambda$.
The proof that when $\lambda >0$, $\alpha_{\rm c}(\lambda)$ can be chosen being strictly less than the trivial bound $\log \mu$ and that it is decreasing with respect to $\lambda$ in $[0, \infty)$ requires some work.
We will first state and prove an auxiliary lemma and, after that, we will present the proof.

For any 
$A \subseteq \Lambda_L$, define 
$\Omega_A$ as the set of configurations $\boldsymbol{\gamma} \in \Omega$ such that $\gamma_x \cap (\Lambda_L \setminus A) = \emptyset$ for all $x \in \Lambda_L$. In other words, all self-avoiding polygons $\gamma_x$ such that $x \not\in A$ are empty and all self-avoiding polygons $\gamma_x$ such that $x \in A$ are either empty or contained in $A$.
Define also,
$$
Z(A) : = \sum\limits_{\boldsymbol{\gamma} \in \Omega_A} e^{- \mathcal{H}_{L, \alpha, \lambda}(\boldsymbol{\gamma})},
$$
where the dependence of the previous quantity on $L$, $\alpha$ and $\lambda$ is implicit.
\begin{lemma}\label{lemma:estimatepartition}
Let $A \subset B \subseteq \Lambda_L$ be  sets such that $A$  contains at least $h$ distinct pairs of adjacent sites.
In other words $A$ contains at least $h$ pairs of sites,
$\{ x_1, y_1 \}$, $\ldots$, $\{x_h, y_h\}$, such that $x_i$ and $y_i$ are nearest neighbours for each $i \in \{1, \ldots h \}$
and  $\{x_i, y_i\} \cap \{x_j, y_j\} = \emptyset$ if $i \neq j$. Then,
$$
\frac{{Z}(B \setminus A)}{{Z}(B)} \leq \left(\frac{1}{1 + e^{-2\alpha}} \right)^h.
$$
\end{lemma} 
\begin{proof}
We use a simple multi-valued map principle, which consists of filling an empty region of space with polygons consisting in just two edges.
To begin, note that
$$
Z(B) \geq \sum\limits_{\substack{ \gamma \in \Omega_B : \\ \gamma_x \cap A = \emptyset \mbox{ \footnotesize if } x \in B \setminus A \\ \gamma_x \cap B \setminus A  = \emptyset \mbox{ \footnotesize  if } x \in  A
 }} e^{-\mathcal{H}_{L,\alpha, \lambda}(\boldsymbol{\gamma})} 
= Z(B \setminus A) Z(A),
$$
where the inequality holds since the sum is over configurations $\boldsymbol{\gamma}$ in a subset of $\Omega_B$ and the identity holds since the previous sum can be factorised into the product of two independent sums which can be identified with $Z(B\setminus A)$ and $Z(A)$.
From this we deduce that,
\begin{equation}\label{relationpartitions}
\frac{{Z}( B \setminus A)}{{Z}(B)} \leq 
\frac{{Z}( B \setminus A)}{{Z}(B\setminus A) {Z}(A) }
\leq \frac{1}{  {Z}( A)}.
\end{equation}
The proof now consists in providing a lower bound to $Z(A)$.
For this, we will sum over configurations in a subset of $ \Omega_A$ for which we will compute the weight explicitly. To define this set, recall that $A$ contains $h$ disjoint pairs of adjacent sites, $\{x_1, y_1\}$, $\ldots$, $\{x_h, y_h\}$.
Given two sets of integers
$C, D \subset \{1, \ldots, h\}$ such that $C \cap D = \emptyset$, we let 
$\boldsymbol{\gamma}^{C,D} =  (\boldsymbol{\gamma}^{C,D}_x)_{x \in \Lambda_L} \in \Omega_A$ be the configuration which is defined as follows. Namely,  
 for each $i \in C$, $\gamma^{C,D}_{x_i}$ is the polygon consisting in   one edge  directed from $x_i$ to $y_i$ and one edge directed  from $y_i$ to $x_i$, and  $\gamma^{C,D}_{y_i} : = \zeta$. For each $i \in D$, $\gamma^{C,D}_{x_i} : = \zeta$ and $\gamma^{C,D}_{y_i}$ is the polygon consisting in  one edge directed from $y_i$ to $x_i$ and one edge directed  from $x_i$ to $y_i$. For each $x \in A \setminus \cup_{i \in C \cup D} \big ( \{x_i, y_i\} \big ) $, we set $\gamma^{C,D}_x  := \zeta$. This concludes the definition of  
$\boldsymbol{\gamma}^{C,D}$.
By definition of energy,
 (\ref{Hamiltonian alpha lambda}),  we have that for each $C, D \subset \{1, \ldots, h\}$ (possibly empty sets) such that $C \cap D = \emptyset$,
 $$
e^{- \mathcal{H}_{L, \alpha, \lambda} (\boldsymbol{\gamma}^{C,D} )  } = \frac{1}{2^{|C| + |D|}} \,  e^{-2 \alpha (| C | + |D|)}.
 $$
Using the fact that the number of pairs of sets $C, D \subset \{1, \ldots, h\}$ such that $C \cap D = \emptyset $ and $|C| + |D| = n \in \{0, \ldots, h\}$ is $2^n \binom{h}{n}$ (there are $\binom{h}{n}$ ways to choose $h$ out of $n$ pairs and $2^n$ ways to choose which vertex of the pair generates a polygon),
we obtain that, 
\begin{align*}
Z(A)  & \geq \sum\limits_{\substack{C, D \subset \{1, \ldots, h\} :  \\ C \cap D = \emptyset} } e^{- \mathcal{H}_{L, \alpha, \lambda} (\boldsymbol{\gamma}^{C,D} )  } \\
& = \sum\limits_{\substack{C, D \subset \{1, \ldots, h\} :  \\ C \cap D = \emptyset} }
\frac{1}{2^{|C| + |D|}} \,  e^{-2 \alpha (| C | + |D|)} \\
& = \sum\limits_{n \in \{0, \ldots, h\} }
\sum\limits_{\substack{C, D \subset \{1, \ldots, h\} :  \\ C \cap D = \emptyset \\  |C| + |D| = n } }
\frac{1}{2^{n}} \,  e^{-2 \alpha  n}   \\
& = 
\sum\limits_{n \in \{0, \ldots, h\} }  2^{n} \, \, \binom{h}{n}
\frac{1}{2^{n}} \,  e^{-2 \alpha   n}  \\
& = ( 1 + 2^{-2 \alpha})^{h}.
\end{align*}
Replacing the previous bound in (\ref{relationpartitions}) the proof of the lemma is concluded.
\end{proof}
Let
$
H(\boldsymbol{\gamma}) =  \{   x \in \Lambda_L : \, \, \forall  y \in \Lambda_L  \setminus  \{o\},  \, \, x \notin \gamma_y    \}
$
 be the set of sites which
are not visited by any  polygon starting from $\Lambda_L \setminus \{o\}$.
For any $q \in (0, 1)$, we have that
\begin{equation}
\mathbb{P}_{L, \alpha,\lambda} ( \gamma_o =  \tilde{\gamma} ) \nonumber = \\
\mathbb{P}_{L, \alpha, \lambda}  ( \gamma_o =  \tilde{\gamma}, \, \,   |H \, \, \cap \, \, \tilde{\gamma} | <  q  \, \|\tilde{\gamma}\|  \,  )  +
\mathbb{P}_{L, \alpha, \lambda}   ( \gamma_o =  \tilde{\gamma}, \, \,  \ |H \, \, \cap \, \, \tilde{\gamma} | \geq  q  \, \|\tilde{\gamma}\|  \,  ) .
\end{equation}
We will bound both terms on the right-hand side of the previous expression from above separately.
For the first term we have that,
\begin{equation}
\mathbb{P}_{L, \alpha,\lambda}  ( \gamma_o =  \tilde{\gamma}, \, \,  \, \,  |H \, \, \cap \, \, \tilde{\gamma} | <  q  \, \|\tilde{\gamma}\|  \,  ) 
\leq e^{- ( \alpha +  (1-q) \lambda) \| \tilde{\gamma}\| }.
\end{equation}
This  follows from the definition of  $\mathbb{P}_{L, \alpha, \lambda} $.
We now bound  the second term from above. 
For this, we will choose $q$ close to one. Recall that $\zeta$ denotes the empty polygon.
We find
\begin{align}
& \mathbb{P}_{L, \alpha, \lambda} ( \gamma_o =  \tilde \gamma, \, \,  \, \,  |H \, \, \cap \, \, \tilde{\gamma} | \geq  q  \, \|\tilde{\gamma} \|  \,  ) \nonumber
\\ \leq & e^{- \alpha \| \tilde{\gamma} \|} \, \,  \mathbb{P}_{L, \alpha, \lambda}  ( \gamma_o =  \zeta , \, \,  \, \,  |H \, \, \cap \, \, \tilde{\gamma} | \geq  q  \,  \| \tilde{\gamma} \|  \,  ) \nonumber
 \\ = &
e^{- \alpha \| \tilde{\gamma} \|}  \, \sum\limits_{ A \subset \tilde{\gamma}, |A| \geq  q \| \tilde{\gamma} \| } 
\mathbb{P}_{L, \alpha, \lambda}   ( \, \gamma_o = \zeta,   H \, \, \cap \, \, \tilde{\gamma} = A \,  ) .\label{eq:expression1}
\end{align}
Now observe that, for $A \subset \tilde{\gamma}$,
\begin{align}
& \mathbb{P}_{L, \alpha, \lambda}  (   \gamma_o = \zeta, \, H \, \, \cap \, \, \tilde{\gamma} = A   ) \nonumber \\
\leq   &
\mathbb{P}_{L, \alpha, \lambda} \Big  (  \big \{  \forall x \in \Lambda_L \setminus A, \gamma_x \cap A = \emptyset \big   \}  \, \, \cap  \, \, \big  \{  \forall x \in A, \gamma_x = \zeta  \big \} \Big  )  \, \nonumber \\
= &
\frac{{Z}( \Lambda_L \setminus A)}{{Z}( \Lambda_L)}.\label{eq:expression2}
\end{align}
If at least $q \| \tilde{\gamma} \|$ vertices of $\tilde{\gamma}$ are not visited by any polygon, then the worst-case bound  leads to the conclusion that we can find at least $\frac{3q - 2}{2} \,  \|\tilde{\gamma} \|$ non-overlapping pairs of sites $ \{x_i, y_i\} \subset \tilde{\gamma}$,
$i \in [1, \frac{3q - 2}{2} \,  \|\tilde{\gamma} \|]$, which are not visited by any polygon.
Thus, by using (\ref{eq:expression2}) in (\ref{eq:expression1}) and applying Lemma \ref{lemma:estimatepartition},  we obtain that,
\begin{align}
\mathbb{P}_{L, \alpha, \lambda}  ( \gamma_o =  \tilde{\gamma}, \, \,  \, \,  |H \, \, \cap \, \, \tilde{\gamma} | \geq  q  \, \| \tilde{\gamma} \|  \,  ) 
& \leq e^{- \alpha \| \tilde{\gamma} \|}
\left( \frac{1}{  1 + e^{-2 \alpha} } \right)^{\frac{3q - 2}{2} \,  \| \tilde{\gamma} \|} \, \, 
\sum\limits_{n= \lceil q \Vert \tilde{\gamma} \Vert\rceil }^{\| \tilde{\gamma}\|  } \binom{ \| \tilde{\gamma} \|}{n} \\
& \leq  \| \tilde\gamma \| \, \, e^{- \alpha \| \tilde{\gamma} \|}
\left( \frac{1}{  1 + e^{-2 \alpha} } \right)^{\frac{3q - 2}{2} \,  \|\tilde{\gamma}\|} \, e^{(1-q) \, [ \, 1 - \log (1-q) \, ] \,  \| \tilde\gamma \|} ,
\end{align}
where we used 
$\binom{ \| \tilde\gamma \|}{ \lceil q \| \tilde\gamma \| \rceil} = 
\binom{ \| \tilde\gamma \|}{ \| \tilde\gamma \| \, - \,\lceil q \|  \tilde\gamma \| \rceil} \leq \big ( e /(1-q)\big )^{(1-q) \| \tilde\gamma \|}.
$
Putting together the previous bounds, we obtain that, for any $q \in (0, 1)$,  for any $\tilde \gamma \in {\rm SAP}_o$,
\begin{align*}
\mathbb{P}_{L, \alpha, \lambda}  ( \gamma_o =  \tilde{\gamma} )  & \leq e^{- ( \alpha +  (1-q) \lambda) \| \tilde{\gamma}\| } + 
\| \tilde\gamma \| \, \, e^{- \alpha \| \tilde{\gamma} \| - \frac{3q - 2}{2} \,g  \|\tilde{\gamma}\|} e^{(1-q) \, [ \, 1 - \log (1-q) \, ] \,  \| \tilde\gamma \|} \\
& \leq 
e^{- b_1  \| \tilde{\gamma}\| } + 
\| \tilde\gamma \| \, \, e^{-  b_2   \|\tilde{\gamma}\|},
\end{align*}
where we defined,
\begin{align*}
g & = g(\alpha)  : = \log ( 1 + e^{-2 \alpha}), \\
b_1 & = b_1(\alpha, \lambda, q)  := \alpha + (1-q) \lambda \\
b_2 & = b_2(\alpha, q) :=  \alpha   + \frac{3 q - 2}{2} g -
(1-q) [1 - \log (1-q)].
\end{align*}
Our goal now is to prove that $b_1$ and $b_2$ are strictly larger than $\log \mu$ for an appropriate choice of $q$, $\lambda$ and $\alpha < \log \mu$. 
Since the number of polygons of length $k$ grows like $\sim e^{ \log \mu k}$,  this will give exponential decay of the polygon length. 
First, 
\[ 
\text{let $\alpha^*$ be the solution of $\alpha + g(\alpha)/5 = \log \mu$},
\]
 and note that $\alpha^* < \log \mu$ depends only on $d$.
Now we will choose $q$ close enough to one so that  $b_2 > \log \mu$ for any $\alpha > \alpha^*$. At the same time we need a small $q$ since this makes $b_1$ large.
More precisely, we set the value of $q$ as follows,
\begin{equation}\label{eq:choiceq}
q = \min \{q^{\prime} \in [0, 1] \, :  \,  3q^{\prime}/2 - 1 \geq 1/4   \mbox{ and }   (1-q^{\prime})(1 - \log(1-q^{\prime})) \leq g(\alpha^*) / 30  \} 
\end{equation}
and we note that such a $q$ depends only on $d$ and that $q \in (0, 1)$.
Since  the function  $\alpha + g(\alpha)/4$ is non-decreasing in $\alpha$, for any $\alpha > \alpha^*$,
 we find
\begin{align*}
b_2(\alpha, q) & \geq 
\alpha   + \frac{1}{4} g(\alpha) - \frac{1}{30} g(\alpha^*)
\geq 
\alpha^*   + \frac{1}{4} g(\alpha^*) -
 \frac{1}{30}
 g(\alpha^*) \\
 & \geq 
 \alpha^*   + \frac{1}{5} g(\alpha^*) +  \frac{1}{20}
 g(\alpha^*) -
 \frac{1}{30}
 g(\alpha^*) 
  \geq \log \mu + \frac{1}{60} 
 g(\alpha^*),
\end{align*}
Then, 
for any $k \in \mathbb{N}$,
\begin{align*}
\mathbb{P}_{L, \alpha, \lambda} ( \| \gamma_o \| > k ) &  \leq \sum\limits_{m=k+1}^{\infty} \mathbb{P}_{L, \alpha, \lambda} (  \|\gamma_o\|  = m)  \\
& \leq \sum\limits_{m=k+1}^{\infty} \Big ( |{\rm SAP}_{o,m}|  \, \, e^{-b_1 \,  m } 
+   |{\rm SAP}_{o,m}|   \, m \, e^{- b_2 \,  m }  \Big )  \\
& \leq 2 (d-1) \, \sum\limits_{m=k+1}^{\infty} m  e^{(- b_1 + \log \mu) m }
+
2 (d-1) \, \sum\limits_{m=k+1}^{\infty} m^2 e^{(- b_2 + \log \mu) m} \\
& \leq  
2 (d-1) \, \sum\limits_{m=k+1}^{\infty} m  e^{(- \alpha - (1-q) \lambda + \log \mu) m }
+
2 (d-1) \, \sum\limits_{m=k+1}^{\infty} m^2 e^{- \frac{g(\alpha^*)}{60} m}.
\end{align*}
Now set
\begin{equation}
\label{alphaCformula}
\alpha_c(\lambda) : = \big (   \, 
\log \mu - (1-q) \lambda \, 
 \vee   \,  \alpha^* \big ),
\end{equation}
From the previous bound we obtain that  for any $\lambda \geq 0$ and $\alpha > \alpha_c(\lambda)$ there exist two constants $\delta, C \in (0, \infty)$ such that,
for any $k \in \mathbb{N}$,
$$
\mathbb{P}_{L, \alpha, \lambda} ( \| \gamma_o \| > k )  \leq
C \, e^{-   \delta \, k}.
$$
Since $\alpha_c(\lambda) < \log \mu$ for any $\lambda \in (0, \infty)$, this concludes the proof.

\subsection{Proof of Theorem \ref{theo:weakly space filling}}
\label{sect:proof of Theo3}

In this section, we show that there is a regime in which a self-avoiding polygon is infinite in the infinite-volume limit and space-filling in a quantifiable sense. 
We will  adopt the strategy developed in  \cite{Copin}, which consists of showing that, under appropriate assumptions, it is much more likely that a self-avoiding polygon enters a certain large area than that it gets only close to it without entering it. 
The way to get there is to consider, for any self-avoiding polygon $\gamma_o$ getting close to the area $F$ without entering it (see also Figure \ref{Fig:extensions}), the cumulative weight of certain extensions  $\tilde  \gamma$  entering the area.
\begin{figure}
\centering
\includegraphics[scale=0.42]{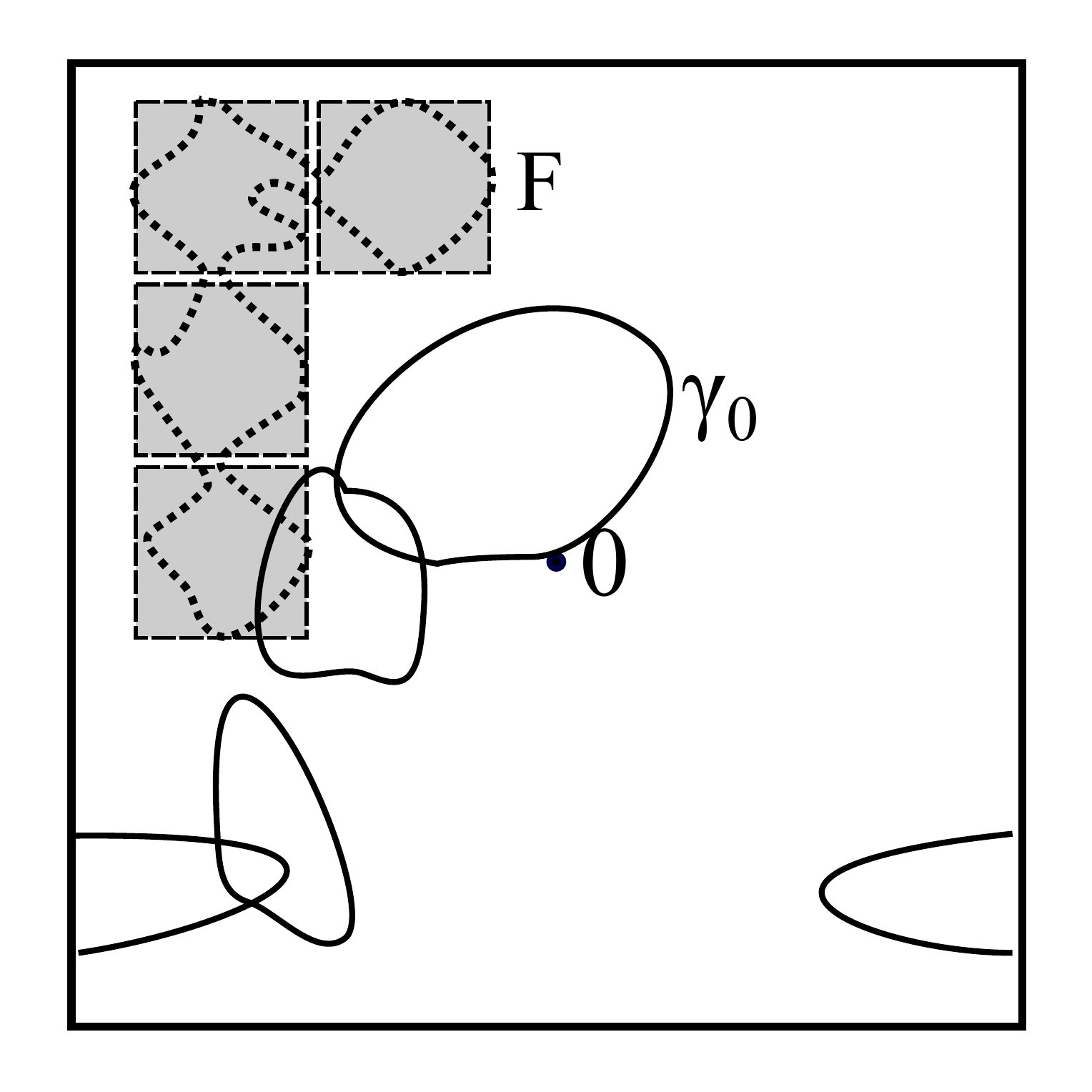}
\caption{The dotted path represents a self-avoiding polygon $\gamma_F$ touching all midpoints of the boundary sides of the boxes,   $F$ is the collection of boxes in the figure.  The self-avoiding polygon $\gamma_o$ does not enter $F$, but it reaches a neighbouring box of $F$.}
\label{Fig:extensions}
\end{figure}
One needs to show that the ``typical'' weight of such extensions is not less than  $ e^{- c \Vert \tilde \gamma \Vert}$ for some $c < \log \mu$. Since the number of the extensions of length $n$ grows like $\sim \mu^n$, this  guarantees that the entropy beats the energy cost for constructing such extensions.  
This strategy is implemented in the proof of Proposition \ref{prop: sufficient condition}, which uses a multi-valued map principle.
We will start by giving an outline of the approach developed in \cite{Copin} and state Proposition \ref{prop: sufficient condition}, from which we deduce the theorem.
Then we will prove that the assumptions of Proposition \ref{prop: sufficient condition} hold in the relevant cases.

As in \cite{Copin}, our presentation of the method will be restricted to the 2-dimensional case. Generalizations to higher dimensions are possible along the lines of \cite[Remark 8]{Copin}. We now start to introduce the key concepts from \cite{Copin}.
 To facilitate calculations, we introduce the notion of boxes and require that relevant extensions respect their structure, which allows us to apply Proposition \ref{prop:Zm} and Lemma \ref{lem:supermult} below.

\begin{definition}\label{def:Pm}
Let $P_m$ be the set of self-avoiding polygons in the square box $[0,2m+1]^2\subset \Lambda_L$ that touch the middle of every face of the square, i.e., they contain the edges $((m,0),(m+1,0))$, $((2m+1,m),(2m+1,m+1))$, $((m,2m+1),(m+1,2m+1))$, and $((0,m),(0,m+1))$. Let further $Z_m(x)=\sum_{\gamma\in P_m}\frac{x^{\left\Vert\gamma\right\Vert}}{\left\Vert\gamma\right\Vert \vee 1}$ for any $x>0$.
\end{definition}
Note that our definition deviates from the one given in \cite{Copin} by introducing the denominator $\left\Vert\gamma\right\Vert \vee 1$, which does not fundamentally change the behaviour of $Z_m(x)$, however. We therefore arrive at

\begin{proposition}(\cite[Proposition 3]{Copin})\label{prop:Zm}
Let $x>1 / \mu$. Then $\limsup_{m\to\infty}Z_m(x)=\infty$.
\end{proposition} 

\begin{definition}
Let $m>0$. The cardinal edges of a square box $B\subset\Lambda_L$ with side length $2m+1$ are the edges in the middle of every face (cf. Definition \ref{def:Pm}). Two such boxes $B$ and $B^{\prime}$ are called adjacent if they are disjoint and if there are cardinal edges $(x,y)$ of $B$ and $(w,z)$ of $B^{\prime}$ such that $x\sim w$ and $y\sim z$. A family $F$ of such boxes is connected if every two boxes in $F$ can be connected by a path of adjacent boxes in $F$ (see \cite[Figure 5]{Copin}). Let $\mathcal{F}\left(\Lambda_L,m \right)$ be the set of connected families of boxes of side length $2m+1$ included in $\Lambda_L$. For $F\in\mathcal{F}\left(\Lambda_L,m \right)$, let $\mathcal{V}_F$ be the set of vertices of boxes in $F$ and $\mathcal{E}_F$ be the set of edges with both end-points in $\mathcal{V}_F$. Let $\mathcal{EC}_F$ the set of external cardinal edges of $F$ and $S_F$ the set of self-avoiding polygons in $\mathcal{E}_F$ that visit all edges in $\mathcal{EC}_F$. Let further $Z_F \left(x\right)=\sum_{\gamma\in S_F}\frac{x^{\left\Vert\gamma\right\Vert}}{\left\Vert\gamma\right\Vert \vee 1}$. For the sake of simplicity, we fix coordinates $\left\{0,1,2,...,L-1\right\}^2$ for $\Lambda_L$ and only consider boxes with their lower left corner in $(2m+2)\mathbb{Z}^2$ with respect to these coordinates. For two subsets $A,B\subset \Lambda_L$, we then define the box distance (denoted by $\mathrm{boxdist}$) between them as the size of the smallest set of connected boxes containing one vertex in $A$ and one in $B$ minus $1$.
\end{definition}

\begin{lemma}(Claim in the proof of \cite[Proposition 7]{Copin})\label{lem:supermult}
Let $F\in\mathcal{F}(\Lambda_L,m)$ and $m\geq 2$. Then $Z_F(x)\geq Z_m(x)^{\left\vert F\right\vert}$, where $\vert F\vert$ denotes the cardinality of $F$.
\end{lemma}

\begin{proof}
Taking note of $\frac{1}{m_1+m_2}\geq\frac{1}{m_1\,m_2}$ for $m_1,m_2\geq 2$, one sees that the proof in \cite{Copin} still works.
\end{proof}

\begin{definition}\label{def:concatenate}
Let $\tilde{\gamma}_1$ and $\tilde{\gamma}_2$ be self-avoiding polygons such that either $\mathcal{V}(\tilde{\gamma}_1)\cap\mathcal{V}(\tilde{\gamma}_2)=\left\{v_1,v_2\right\}$ and  $\mathcal{E}(\tilde{\gamma}_1)\cap\mathcal{E}(\tilde{\gamma}_2)=\{(v_1,v_2)\}$ or $\mathcal{V}(\tilde{\gamma}_1)\cap\mathcal{V}(\tilde{\gamma}_2)=\left\{v_1,v_2,v_3\right\}$ and  $\mathcal{E}(\tilde{\gamma}_1)\cap\mathcal{E}(\tilde{\gamma}_2)=\{(v_1,v_2),(v_2,v_3)\}$ with distinct vertices $v_1,v_2,v_3$ hold. We then define $\tilde{\gamma}_1\sqcup\tilde{\gamma}_2$ via its edges $\mathcal{E}(\tilde{\gamma}_1\sqcup\tilde{\gamma}_2):=\mathcal{E}(\tilde{\gamma}_1)\Delta\mathcal{E}(\tilde{\gamma}_2)$, where $\Delta$ denotes the symmetric difference.
\end{definition}
Note that $\tilde{\gamma}_1\sqcup\tilde{\gamma}_2$ is again a self-avoiding polygon with either $\Vert\tilde{\gamma}_1\sqcup\tilde{\gamma}_2\Vert =\Vert\tilde{\gamma}_1\Vert +\Vert\tilde{\gamma}_2\Vert-2$  or $\Vert\tilde{\gamma}_1\sqcup\tilde{\gamma}_2\Vert =\Vert\tilde{\gamma}_1\Vert +\Vert\tilde{\gamma}_2\Vert-4$.

Let further $\Theta_F$ be the set of self-avoiding polygons containing the origin which do not intersect $F$, but reach a neighbouring box. For $\tilde{\gamma}_o\in\Theta_F$, let $e$ be an external cardinal edge of $F$ in box distance $1$ from $\tilde{\gamma}_o$ and let $l(\tilde{\gamma}_o)$ be a self-avoiding polygon such that
\begin{itemize}
\item $\mathcal{E}_F \cap \mathcal{E}(l(\tilde{\gamma}_o))=\{e\}$,
\item $l(\tilde{\gamma}_o)$ intersects $\tilde{\gamma}_o$ either at exactly one edge or at exactly two adjacent edges, and
\item $l(\tilde{\gamma}_o)$ has length smaller than $100m$.
\end{itemize}
It can be checked that such a polygon always exists (cf. \cite[p.10]{Copin}). Now define $f:\Theta_F \times S_F\longrightarrow {\rm SAP}_o$ by
$$
f\left(\tilde{\gamma}_o,\tilde{\gamma}_F \right)= \tilde{\gamma}_o \sqcup \left( \tilde{\gamma}_F \sqcup l(\tilde{\gamma}_o)\right).
$$ 
It can be shown (see \cite{Copin}) that $f$  is at most $100m^2$ to $1$.
See also Figure \ref{Fig:extensions} for a depiction of typical examples of $\tilde{\gamma}_o$ and $\tilde{\gamma}_F$.

We start by identifying an inequality that is sufficient for proving our result. 
We will then show the validity of that inequality in the different regimes by 
different methods.

\begin{proposition}\label{prop: sufficient condition}
Let $\alpha, \lambda \in \bbR$. Assume that we can find 
$x > 1/\mu$ and  $C>0$ such that 
\begin{equation}
	\label{eq:assumption}
	\mathbb{P}_{L,\alpha,\lambda}\left[\gamma_o=f\left(\tilde{\gamma}_o,\tilde{\gamma}_F\right)\right] 
\geq
C\, \mathbb{P}_{L,\alpha,\lambda}\left[\gamma_o=\tilde{\gamma}_o\right] \frac{x^{\left\Vert\tilde{\gamma}_F\right\Vert}}{\left\Vert\tilde{\gamma}_F\right\Vert}
\end{equation}
holds for all $L>0$, and for all $F\in\mathcal{F}\left(\Lambda_L,m\right)$, $\tilde{\gamma}_o\in\Theta_F$ and $\tilde{\gamma}_F\in S_F$ with large enough $m$. Then equation \eqref{eq: weakly space filling} holds for this pair 
$(\alpha,\lambda)$. 
\end{proposition}

\begin{proof}
Since $f$ is at most $100m^2$ to $1$, we have
\begin{equation}\label{eq:boxdist}
1  \geq \mathbb{P}_{L,\alpha,\lambda}\left[\gamma_o \in f\left(\Theta_F \times S_F\right)\right]
 \geq \frac{1}{100m^2}\sum_{\tilde{\gamma}_o\in\Theta_F ,\tilde{\gamma}_F\in S_F} \mathbb{P}_{L,\alpha,\lambda}\left[\gamma_o=f\left(\tilde{\gamma}_o,\tilde{\gamma}_F\right)\right].
\end{equation}
Using assumption \eqref{eq:assumption} and summing up yields
\begin{align*}
\sum_{\tilde{\gamma}_o\in\Theta_F ,\tilde{\gamma}_F\in S_F} \mathbb{P}_{L,\alpha,\lambda}\left[\gamma_o=f\left(\tilde{\gamma}_o,\tilde{\gamma}_F\right)\right]
&\geq C\, \sum_{\tilde{\gamma}_o\in\Theta_F} \mathbb{P}_{L,\alpha,\lambda}\left[\gamma_o=\tilde{\gamma}_o\right] 
\sum_{\tilde{\gamma}_F\in S_F} \frac{x^{\left\Vert\tilde{\gamma}_F\right\Vert}}{\left\Vert\tilde{\gamma}_F\right\Vert}\\
&\geq
C\, \mathbb{P}_{L,\alpha,\lambda}\left[\mathrm{boxdist}\left(\gamma_o,\mathcal{V}_F\right)=1\right]\, Z_F\left(x\right).
\end{align*}
Equation \eqref{eq:boxdist} and Lemma \ref{lem:supermult} then lead to
\begin{align*}
\mathbb{P}_{L,\alpha,\lambda}\left[\mathrm{boxdist}\left(\gamma_o,\mathcal{V}_F\right)=1\right] 
\leq
\tilde{C}\, Z_m\left(x\right)^{-\left|F\right|}
\end{align*}
for each $F\in\mathcal{F}\left(\Lambda_L,m\right)$, with $\tilde C$ only depending on $C$ and $m$.

Recall that $\Gamma_{L}^{\xi}(\gamma_o)$ denotes the set of vertices which have distance at least $\xi$ from
$\gamma_o$. Following the proof of \cite[Theorem 6]{Copin} in applying Proposition \ref{prop:Zm}, we obtain that there is $m\in\mathbb{N}$, $\tilde{c}>0$ as well as $C=C(x,m)>0$ such that for every $N>0$,
\[
\mathbb{P}_{L,\alpha,\lambda}\left[\text{there is a component of }\Lambda_L \setminus \Gamma_L^{6m}\left(\gamma_o\right)\text{ of size }>N\right]\leq C L^2\, e^{-\tilde{c} \,N}.
\]
If we now choose $\xi:=6m$ and suitable $c>0$, then \eqref{eq: weakly space filling} is proved.
\end{proof}

For the proof of inequality \eqref{eq:assumption} we need to distinguish between two cases:  The situation is easier if we assume that $\alpha<\log\mu$. Then we only have to ascertain that the interaction between polygons does not fundamentally change the picture as self-avoiding polygons with weights $e^{-\alpha}$ are space-filling in this case (see \cite[Theorem 1]{Copin}). On the other hand, if $\alpha\geq\log\mu$, we will have to show that there is sufficient attractive interaction with polygons at sites $x\neq o$ to compensate for the high value of $\alpha$. 

The case $\alpha < \log \mu$ of Theorem \ref{theo:weakly space filling} 
is proved by the following Lemma:

\begin{lemma}\label{prop:gluingsmall}
Let $\alpha<\log\mu$ and $\lambda<\log\mu-\alpha$. Then there exist $x > 1/\mu$ 
and $C>0$ so that \eqref{eq:assumption} holds. 
\end{lemma}

\begin{proof}
Let $\alpha < \log \mu$ and $0 < \lambda < \log \mu - \alpha$, and assume 
that the self-avoiding polygons $\tilde{\gamma}_o$, $\tilde{\gamma}_v$ satisfy
the assumptions of Definition \ref{def:concatenate}.  
Then, by supposing interaction at maximally many sites, we obtain the bound
\begin{align*}
\mathbb{P}_{L,\alpha,\lambda}\left[\gamma_o=\tilde{\gamma}_o\sqcup\tilde{\gamma}_v\right] 
&\geq 
\left(e^{2\alpha}\wedge e^{3\alpha}\right) \,\mathbb{P}_{L,\alpha,\lambda}\left[\gamma_o=\tilde{\gamma}_o\right]\frac{\left\Vert\tilde{\gamma}_o\right\Vert}{\left\Vert\tilde{\gamma}_o\right\Vert+\left\Vert\tilde{\gamma}_v\right\Vert-2}e^{-\left(\alpha+\lambda\right)\left\Vert \tilde{\gamma}_v\right\Vert}\\
&\geq 
\left(e^{2\alpha}\wedge e^{3\alpha}\right) \,\mathbb{P}_{L,\alpha,\lambda}\left[\gamma_o=\tilde{\gamma}_o\right]\frac{e^{-\left(\alpha+\lambda\right)\left\Vert \tilde{\gamma}_v\right\Vert}}{\left\Vert\tilde{\gamma}_v\right\Vert}.
\end{align*}
If $\lambda \leq 0$, we obtain in a similar way that 
\[
\mathbb{P}_{L,\alpha,\lambda}\left[\gamma_o=\tilde{\gamma}_o\sqcup\tilde{\gamma}_v\right] 
\geq 
\left(e^{2\alpha}\wedge e^{3\alpha}\right)e^\lambda\, \mathbb{P}_{L,\alpha,\lambda}\left[\gamma_o=\tilde{\gamma}_o\right]\frac{e^{-\alpha\left\Vert \tilde{\gamma}_v\right\Vert}}{\left\Vert\tilde{\gamma}_v\right\Vert}.
\]
%
%
%
The additional factor of $e^\lambda$ takes into account the second case in Definition \ref{def:concatenate} in which the vertex $v_2$ does not belong to $\tilde{\gamma}_o\sqcup\tilde{\gamma}_v$, but to $\tilde{\gamma}_o$.
Thus with   
\[
x = \exp(- \alpha - (\lambda \vee 0)),
\]
we find that for all $\alpha \in \bbR$, all $\lambda < \log \mu - \alpha$, 
all $F \in \caF(\Lambda_L,m)$ with sufficiently large $m$, all $\tilde{\gamma}_o\in\Theta_F$ and all $\tilde{\gamma}_F\in S_F$, we can 
define  $\tilde{\gamma}_v=\tilde{\gamma}_F \sqcup l(\tilde{\gamma}_o)$ 
and compute 
\begin{align*}
\mathbb{P}_{L,\alpha,\lambda}\left[\gamma_o=f\left(\tilde{\gamma}_o,\tilde{\gamma}_F\right)\right] 
& \geq 
C^\prime\, \mathbb{P}_{L,\alpha,\lambda}\left[\gamma_o=\tilde{\gamma}_o\right]x^{\left\Vert\tilde{\gamma}_v\right\Vert-\left\Vert \tilde{\gamma}_F\right\Vert}\frac{x^{\left\Vert\tilde{\gamma}_F\right\Vert}}{\left\Vert\tilde{\gamma}_v\right\Vert}\\
& =
C^\prime\, \mathbb{P}_{L,\alpha,\lambda}\left[\gamma_o=\tilde{\gamma}_o\right]x^{\left\Vert l\left(\tilde{\gamma}_o\right)\right\Vert -2}\frac{x^{\left\Vert\tilde{\gamma}_F\right\Vert}}{\left\Vert l\left(\tilde{\gamma}_o\right)\right\Vert +\left\Vert\tilde{\gamma}_F\right\Vert -2}\\
&\geq 
C\, \mathbb{P}_{L,\alpha,\lambda}\left[\gamma_o=\tilde{\gamma}_o\right] \frac{x^{\left\Vert\tilde{\gamma}_F\right\Vert}}{\left\Vert\tilde{\gamma}_F\right\Vert},
\end{align*}
where $C' = \left(e^{2\alpha}\wedge e^{3\alpha}\right)(\e{\lambda} \wedge 1\, )$.  
In the last line applies the properties of $ l\left(\tilde{\gamma}_o\right)$ and $C$ therefore only depends on $\alpha,\lambda$, and $m$. 
\end{proof}

\noindent {\em Remark:} If we consider $\tilde{\mathbb{P}}_{L,\rho,\nu}$ instead of $\mathbb{P}_{L,\alpha,\lambda}$, the case given by $\rho<\log\mu$ and $\nu\geq 0$ (which is included in Lemma \ref{prop:gluingsmall}) could also have been proved in the following way: For $\nu =0$ the polygons are independent and \cite{Copin} yields that they are weakly space filling. One further easily shows $\lim_{L\to\infty}\tilde{\mathbb{P}}_{L,\rho,0}\left[\mathcal{N} =0\right]=1$, which also holds for $\nu>0$ due to monotonicity of $\tilde{\mathbb{E}}_{L,\rho,\nu}[\mathcal{N}]$. Hence, $\gamma_o$ behaves asymptotically as in the case of independence.\\

In order to treat the case $\alpha\geq\log\mu$, we have to work somewhat harder. The general strategy of the proof is the same as before.
For any finite set $A \subset \Lambda_L \setminus \{o\}$ and $\tilde{\gamma}_o \in {\rm SAP}_o$, let 
$\Omega^{A}_{o,\tilde{\gamma}_o} \subset \prod_{x \in \Lambda_L} \, \, {\rm SAP}_x$
be the set of configurations with $\gamma_x = \zeta$ for all $x\in A$ and $\gamma_o = \tilde{\gamma}_o$. Moreover, write $\mathbb{P}^{\tilde{\gamma}_o}_{L,\alpha,\lambda}  \left[   \cdot \right] := \mathbb{P}_{L,\alpha,\lambda}\left[  \cdot  \left|\gamma_o =\tilde{\gamma}_o\right. \right]$. 

\begin{lemma}\label{lemma:exponential bound}
Let $\alpha \in \mathbb{R}$ and $0<q < 1$. Then there is $\lambda_0\leq0$ such that for all $\lambda<\lambda_0$ and
for any finite set $A \subset\Lambda_{L}\setminus \{o\}$, $L \in \mathbb{N}$ and $\tilde{\gamma}_o\in {\rm SAP}_o$, we have

\[
\mathbb{P}^{\tilde{\gamma}_o}_{L,\alpha,\lambda}\left[\gamma_{x}=\zeta\text{ for all }x\in A\right] \leq q^{\left|A\right|}.
\]
More specifically, $\lambda_0=\log \left(q/2\right) -2\alpha$.
\end{lemma}

\begin{proof}
The proof proceeds via induction on the cardinality of $A$. The main
idea is to construct an injective map $g_{A}:\Omega^{A}_{o,\tilde{\gamma}_o}\rightarrow\Omega^{\emptyset}_{o,\tilde{\gamma}_o}$
such that $e^{-\mathcal{H}\left(\boldsymbol{\gamma}\right)}\leq q^{\left|A\right|}e^{-\mathcal{H}\left(g_{A}\left(\boldsymbol{\gamma}\right)\right)}$
for all $\boldsymbol{\gamma}\in\Omega^{A}_{o,\tilde{\gamma}_o}$ because then
\begin{align*}
& \mathbb{P}^{\tilde{\gamma}_o}_{L,\alpha,\lambda}\left[\gamma_{x}=\zeta\text{ for all }x\in A\right]\\
= &
\frac{\sum_{\boldsymbol{\gamma}\in\Omega^{A}_{o,\tilde{\gamma}_o}}e^{-\mathcal{H}\left(\boldsymbol{\gamma}\right)}}{\sum_{\boldsymbol{\gamma}\in\Omega^{\emptyset}_{o,\tilde{\gamma}_o}}e^{-\mathcal{H}\left(\boldsymbol{\gamma}\right)}}
\leq
\frac{\sum_{\boldsymbol{\gamma}\in\Omega^{A}_{o,\tilde{\gamma}_o}}e^{-\mathcal{H}\left(\boldsymbol{\gamma}\right)}}{\sum_{\boldsymbol{\gamma}\in\Omega^{A}_{o,\tilde{\gamma}_o}}e^{-\mathcal{H}\left(g_A \left(\boldsymbol{\gamma}\right)\right)}} 
\leq
\frac{\sum_{\boldsymbol{\gamma}\in\Omega^{A}_{o,\tilde{\gamma}_o}}e^{-\mathcal{H}\left(\boldsymbol{\gamma}\right)}}{\sum_{\boldsymbol{\gamma}\in\Omega^{A}_{o,\tilde{\gamma}_o}}q^{-\left|A\right|}e^{-\mathcal{H}\left(\boldsymbol{\gamma}\right)}} = q^{\left|A\right|}
\end{align*}
follows. Let $\lambda<\log\left(q/2\right)-2\alpha$. Let first
$A=\left\{ x_{1}\right\} $. We distinguish two cases: Either $x_{1}$
has a neighbour $y_{1}\neq o$ with $\gamma_{y_{1}}\neq\zeta$
or all of its neighbours $y\neq o$ satisfy $\gamma_{y}=\zeta$.
In the first case, let $g_{A}\left(\boldsymbol{\gamma}\right)$ be
obtained from $\boldsymbol{\gamma}$ by substituting $\left\{ x_{1},y_{1}\right\} $
for $\gamma_{x_{1}}=\zeta$, i.e.
\begin{equation*}
\left(g_A(\boldsymbol{\gamma})\right)_z = 
\begin{cases}
\gamma_z & \text{if }z\neq x_1, \\
\{x_1,y_1\}  & \text{if }z=x_1.
\end{cases}
\end{equation*}
Then, by the definition
of $\mathcal{H}$ and $\lambda<0$, we have $e^{-\mathcal{H}\left(g_{A}\left(\boldsymbol{\gamma}\right)\right)}\geq\frac{e^{-2\alpha-\lambda}}{2}e^{-\mathcal{H}\left(\boldsymbol{\gamma}\right)}\geq \frac{1}{q}e^{-\mathcal{H}\left(\boldsymbol{\gamma}\right)}$.
If, on the other hand, $\gamma_{y}=\zeta $ for all $o\neq y\sim x_{1}$,
pick one such neighbour $y_{1}$ of $x_{1}$ and fix a cycle $\gamma^{x_{1},y_{1}}$
of length $4$ visiting both $x_{1}$ and $y_{1}$. Let $g_{A}\left(\boldsymbol{\gamma}\right)$
be obtained from $\boldsymbol{\gamma}$ by substituting $\gamma^{x_{1},y_{1}}$
for both $\gamma_{x_{1}}$ and $\gamma_{y_{1}}$ (only varying the
starting point). Then, $e^{-\mathcal{H}\left(g_{A}\left(\boldsymbol{\gamma}\right)\right)}\geq\frac{e^{-8\alpha-4\lambda}}{16}e^{-\mathcal{H}\left(\boldsymbol{\gamma}\right)}\geq \frac{1}{q}e^{-\mathcal{H}\left(\boldsymbol{\gamma}\right)}$,
and the first step is proved.

Let now $\left|A\right|>1$ and consider
$\boldsymbol{\gamma}\in\Omega^{A}_{o,\tilde{\gamma}_o}$. Choose $x_A\in A$
such that it has a neighbour which does not lie in $A\cup\left\{o\right\}$. If one such neighbour $y_A$ of $x_A$ satisfies $\gamma_{y_{A}}\neq\zeta $,
define $h_{A}\left(\boldsymbol{\gamma}\right)$ by substituting $\left\{ x_{A},y_{A}\right\} $
for $\gamma_{x_{A}}$ as in the previous case (pick one of them arbitrarily if there are several ones). If $\gamma_{y}=\zeta $ for all neighbours $y\notin A\cup\left\{o\right\}$ of $x_A$, denote one of them by $y_A$ and
fix again a cycle $\gamma^{x_{A},y_{A}}$ of length $4$ visiting
both $x_{A}$ and $y_{A}$. Let $h_{A}\left(\boldsymbol{\gamma}\right)$
be obtained by substituting $\gamma^{x_{A},y_{A}}$ for both $\gamma_{x_{A}}$
and $\gamma_{y_{A}}$. Note that $h_{A}:\Omega^{A}_{o,\tilde{\gamma}_o}\rightarrow\Omega^{A\setminus \{x_A\}}_{o,\tilde{\gamma}_o}$
is injective and satisfies $e^{-\mathcal{H}\left(h_{A}\left(\boldsymbol{\gamma}\right)\right)}\geq\frac{1}{q}e^{-\mathcal{H}\left(\boldsymbol{\gamma}\right)}$.
The claim now follows from $g_{A}:=g_{A\setminus\left\{ x_{A}\right\} }\circ h_{A}$. 
\end{proof}

For any set $A \subset \Lambda_L$, let 
$N_A ( \boldsymbol{\gamma} ) : = \sum_{x \in \Lambda_L} \mathbbm{1} \big \{ \gamma_x = \zeta  \big  \}$ be the  number of sites
of $A$ which generate a  polygon  of length zero.
We have that

\begin{lemma}\label{prop:NA bound}
Let $c \in (0, 1)$ and $\alpha \in \mathbb{R}$ be arbitrary. Then, if $\lambda$ is small enough, there exists a constant $K_0>0$ such that for any
$L \in \mathbb{N}$, $A \subset \Lambda_L\setminus\{o\}$ and $\tilde{\gamma}_o\in {\rm SAP}_o$,
$$
\mathbb{P}^{\tilde{\gamma}_o}_{L,\alpha,\lambda} \left[   N_A \geq     c  \left| A\right|  \right] \leq K_0 \left(1 - c \right)^{\left|A\right|}.
$$
More specifically, we need $\lambda\leq\min\left\{0,\log c+\frac{\log(1-c)}{c}-1-\log 2-2\alpha\right\}$.
\end{lemma}

\begin{proof}
Let $q=c\,\exp\left(\frac{\log(1-c)}{c}-1\right)$. Then $\left(\frac{eq}{c}\right)^c = 1-c$ and $\lambda \leq \log \left(q/2\right) -2\alpha$. 
So, by Lemma \ref{lemma:exponential bound},

\begin{align*}
\mathbb{P}^{\tilde{\gamma}_o}_{L,\alpha,\lambda} \left[  N_A \geq  c \left| A \right| \right]
 & \leq    \sum\limits_{n=\left\lceil  c \left|A\right| \right\rceil}^{ \left|A\right|} \,  \binom{ |A| }{ n } q^{ n}
\leq  \sum\limits_{n=\left\lceil  c \left|A\right| \right\rceil}^{ \left|A\right|} \left(\frac{e\left|A\right|q}{c \left|A\right|}\right)^n\\
& \leq  \left(\frac{eq}{c}\right)^{\left\lceil  c \left|A\right| \right\rceil}\frac{1}{1-\frac{eq}{c}}
\leq K_0  \left[\left(\frac{eq}{c}\right)^c\right]^{\left|A\right|}
\leq  K_0 (1-c)^{|A|}.
\end{align*} 
Here, the first inequality applies the union bound and the second inequality is a consequence of the general property of the binomial coefficient that
\[
\binom{m}{n} \leq \left(\frac{\mathrm{e} m}{n}\right)^n
\]
for all $m \geq n$.
\end{proof}

\begin{proposition}\label{prop:gluing}
Let $\alpha\in\mathbb{R}$ and $c\in (0,1)$. Then, if 
$$\lambda\leq\min\left\{0,\log c+\frac{\log(1-c)}{c}-1-\log 2-2\alpha\right\},$$
there is $C_0\in (0,1)$ and there exists $m_0\in\mathbb{N}$ such that for all self-avoiding polygons $\tilde{\gamma}_o$, $\tilde{\gamma}_v$ with $\left\Vert \tilde{\gamma}_v\right\Vert>m_0$ satisfying the assumptions of Definition \ref{def:concatenate} and for any $L\in\mathbb{N}$, we have
\[
\mathbb{P}_{L,\alpha,\lambda}\left[\gamma_o=\tilde{\gamma}_o\sqcup\tilde{\gamma}_v\right] 
\geq 
C_0\, \mathbb{P}_{L,\alpha,\lambda}\left[\gamma_o=\tilde{\gamma}_o\right]\frac{e^{-(\alpha+c\lambda)\left\Vert\tilde{\gamma}_v\right\Vert}}{\left\Vert\tilde{\gamma}_v\right\Vert}.
\]
\end{proposition}

\begin{proof}
Let $A=\mathcal{V}(\tilde{\gamma}_v)\setminus \mathcal{V}(\tilde{\gamma}_o)$. For any $c\in(0,1)$ and $\lambda<0$, we have that
\begin{align*}
\mathbb{P}_{L,\alpha,\lambda}\left[\gamma_o=\tilde{\gamma}_o\sqcup\tilde{\gamma}_v\right]
&\geq   \mathbb{P}_{L,\alpha,\lambda}\left[\gamma_o=\tilde{\gamma}_o\sqcup\tilde{\gamma}_v,N_A<(1-c)\left\vert A\right\vert\right] \\
&\geq 
(e^{2\alpha}\wedge e^{3\alpha})e^\lambda\,\mathbb{P}_{L,\alpha,\lambda}\left[\gamma_o=\tilde{\gamma}_o,N_A<(1-c)\left\vert A\right\vert\right]
\frac{\left\Vert\tilde{\gamma}_o\right\Vert}{\left\Vert\tilde{\gamma}_o\right\Vert+\left\Vert\tilde{\gamma}_v\right\Vert-2}e^{-(\alpha+c\lambda)\,\left\Vert\tilde{\gamma}_v\right\Vert}\\
&\geq
(e^{2\alpha}\wedge e^{3\alpha})e^\lambda\,
\mathbb{P}_{L,\alpha,\lambda}\left[\gamma_o=\tilde{\gamma}_o\right] \,
\mathbb{P}_{L,\alpha,\lambda}^{\tilde{\gamma}_o}\left[N_A<(1-c)\left\vert A\right\vert\right]
\frac{e^{-\left(\alpha+c\,\lambda\right)\left\Vert \tilde{\gamma}_v\right\Vert}}{\left\Vert\tilde{\gamma}_v\right\Vert}.
\end{align*}
If $\lambda\leq -2\alpha+\log c+\frac{\log(1-c)}{c}-1-\log 2$, we may apply Lemma \ref{prop:NA bound} so that
\begin{align*}
\mathbb{P}_{L,\alpha,\lambda}^{\tilde{\gamma}_o}\left[N_A<(1-c)\left\vert A\right\vert\right]
&=1-\mathbb{P}^{\tilde{\gamma}_o}_{L,\alpha,\lambda} \left[   N_A \geq    (1- c)  \left\vert A\right\vert  \right] 
\geq 1-K_0 c^{\left\vert A\right\vert}.
\end{align*}
Now fix $C_0\in (0,1)$ such that $\frac{C_0} {(e^{2\alpha}\wedge e^{3\alpha})e^\lambda}<1$ and choose $m_0\in\mathbb{N}$ such that for all $\left\vert A\right\vert>m_0-2$, we have
$$K_0  c^{\left\vert A\right\vert} \leq 1 - \frac{C_0} {(e^{2\alpha}\wedge e^{3\alpha})e^\lambda}.$$
Then the claim follows.
\end{proof}

The case $\alpha \geq \log \mu$ of Theorem \ref{theo:weakly space filling} is now
completed by 
\begin{lemma} \label{lem:final lemma}
	Let $\alpha \geq \log \mu$ and 
\begin{align*}
\lambda&<0\wedge\sup_{c\in(0,1)}\left\{\left(\frac{\log \mu-\alpha}{c}\right)\wedge\left(-2\alpha+\log c+\frac{\log(1-c)}{c}-1-\log 2\right)\right\}
=:\lambda_{sf}.
\end{align*}
Then there exists $x > 1/\mu$ and $C>0$ so that \eqref{eq:assumption} holds.	
\end{lemma}

\begin{proof}
Let $F \in \caF(\Lambda_L,m)$ for sufficiently large $m$, 
$\tilde{\gamma}_o\in\Theta_F$ and $\tilde{\gamma}_F\in S_F$. Let again $\tilde{\gamma}_v$ be such that $f(\tilde{\gamma}_o,\tilde{\gamma}_F)=\tilde{\gamma}_o\sqcup\tilde{\gamma}_v$. 
Then there are $c\in (0,1)$ and $x>1/\mu$ such that $\exp\left(\alpha +c\lambda\right)>x$. Moreover, we can apply Proposition \ref{prop:gluing} if $m$ is large enough, and we obtain for some $C_0 >0$ that
\begin{align*}
\mathbb{P}_{L,\alpha,\lambda}\left[\gamma_o=f\left(\tilde{\gamma}_o,\tilde{\gamma}_F\right)\right] 
& \geq 
C_0\, \mathbb{P}_{L,\alpha,\lambda}\left[\gamma_o=\tilde{\gamma}_o\right]\frac{x^{\left\Vert\tilde{\gamma}_v\right\Vert}}{\left\Vert\tilde{\gamma}_v\right\Vert}\\
&\geq 
C\, \mathbb{P}_{L,\alpha,\lambda}\left[\gamma_o=\tilde{\gamma}_o\right] \frac{x^{\left\Vert\tilde{\gamma}_F\right\Vert}}{\left\Vert\tilde{\gamma}_F\right\Vert},
\end{align*}
where the last step applies the properties of $l(\tilde{\gamma}_o)$.
The claim is thus proved.
\end{proof}

\section{Appendix}
Here we show how linearity of the energy and an application of H\"older's inequality leads to 
convexity of the pressure.

\begin{lemma}
For any $c \in (0, 1)$, $\rho_1, \rho_2, \nu_1, \nu_2 \in \mathbb{R}$,
$$
\tilde\Phi_L \big ( c \, \rho_1 + (1-c) \rho_2,  c \nu_1 + ( 1-c) \nu_2 \big ) \,  \leq  \, 
c \, \tilde\Phi_L \big ( \rho_1 , \nu_1 \big ) \, \, + \, \, (1-c) \, \tilde\Phi_L \big ( \, \rho_2 ,  \nu_2 \,  \big ) .
$$
\end{lemma}
\begin{proof}
Note first that for any $\boldsymbol{\gamma} \in \Omega$,
$$
\tilde{\mathcal{H}}_{L,c \, \rho_1 + (1-c) \rho_2,  c \nu_1 + ( 1-c) \nu_2}  ( \boldsymbol{\gamma} ) =
c \, \tilde{\mathcal{H}}_{L,\rho_1,   \nu_1 }  \big ( \boldsymbol{\gamma} \big )  + 
(1-c) \,  \tilde{\mathcal{H}}_{L,\rho_2,    \nu_2 }  \big ( \boldsymbol{\gamma}  \big ).
$$
Thus we obtain that,
\begin{align*}
\tilde{Z}_{L,c \, \rho_1 + (1-c) \rho_2,  c \nu_1 + ( 1-c) \nu_2} & =  \sum\limits_{ \boldsymbol{\gamma} \in \Omega} e^{ - \tilde{\mathcal{H}}_{L,c \, \rho_1 + (1-c) \rho_2,  c \nu_1 + ( 1-c) \nu_2}  ( \boldsymbol{\gamma} ) } \\
& = \sum\limits_{ \boldsymbol{\gamma} \in \Omega}   e^{- c \, \tilde{\mathcal{H}}_{L,\rho_1,   \nu_1 }  \big ( \boldsymbol{\gamma} \big ) }
\, \, e^{ - (1-c) \,  \tilde{\mathcal{H}}_{L,\rho_2,    \nu_2 }  \big ( \boldsymbol{\gamma}  \big ) }  \\ 
& \leq  
 \left( 
 \sum\limits_{ \boldsymbol{\gamma}^1 \in \Omega}  
  e^{ -   \tilde{\mathcal{H}}_{L,\rho_1,   \nu_1 }  ( \boldsymbol{\gamma}^1 )  } 
  \right)^{  c  }
  \left( 
  \sum\limits_{ \boldsymbol{\gamma}^2 \in \Omega}  
 e^{ -   \tilde{\mathcal{H}}_{L,\rho_2,    \nu_2 }  ( \boldsymbol{\gamma}^2 )  } 
 \right)^{1-c  } \\
 & = \Big ( \tilde{Z}_{L, \rho_1,   \nu_1}  \Big )^c
 \Big ( \tilde{Z}_{L, \rho_2,  \nu_2} \Big )^{1-c},
\end{align*}
 where in the third  step we used H\"older's inequality with parameters $p=\frac{1}{c}$ and $q= \frac{1}{1-c}$,
 which satisfy $\frac{1}{p} + \frac{1}{q} =1$, and in the last step we used the definition of partition functions.
 Then, by applying the previous inequality to the definition of $\tilde{\Phi}_L$ and by using
 the properties of the logarithm, we conclude the proof.
\end{proof}

\section*{Acknowledgements}
H.S. acknowledges support by Deutsche Telekom Stiftung. L.T. acknowledges support by German Research Foundation (DFG), grant number: BE 5267/1.


\end{document}